\documentclass[11pt]{article}
\usepackage{amsmath,amssymb,amsthm}
\usepackage{url}
\usepackage{hyperref}
\usepackage{graphicx}

\usepackage{fullpage}
\usepackage{enumerate}

\newtheorem{theorem}{Theorem}
\newtheorem{lemma}{Lemma}

\newcommand{\Floor}[1]{\left\lfloor#1\right\rfloor}
\newcommand{\Ceil}[1]{\left\lceil#1\right\rceil}

\title{Enumeration of Payphone Permutations}
\author{Max A. Alekseyev\thanks{The George Washington University, Washington, DC, USA. Email: \href{mailto:maxal@gwu.edu}{maxal@gwu.edu} }}

\begin{document}
\maketitle\thispagestyle{plain}

\begin{abstract}
The desire for privacy significantly impacts various aspects of social behavior as illustrated by people's tendency to seek out the most secluded spot when multiple options are available.
In particular, this can be 
seen at rows of payphones, where people tend to occupy an available payphone that is most distant from those already occupied.
Assuming that there are $n$ payphones in a row and that $n$ people occupy payphones one after another as privately as possible, the resulting assignment of people to payphones defines a permutation, which we will refer to as a \emph{payphone permutation}.
In the present study, we consider different variations of payphone permutations and enumerate them.
\end{abstract}

\section{Introduction}

Social habits often serve as a source of inspiration for exploring novel mathematical objects and properties.
One prominent example is the development of \emph{parking functions}, which originated from a simple observation that people typically have preferences about where to park their cars, and then evolved into a rich combinatorial theory with many applications in other areas of mathematics and computer science~\cite{Yan2015}.

The desire for privacy represents yet another intriguing facet of social behavior, as illustrated by people's tendency to seek out the most secluded spot when multiple options are available. We note some similarities to the concept of \emph{social distancing}, a widely adopted practice to minimize health risks during the COVID-19 pandemic, which generally requires people to maintain some minimum physical separation distance from each other. When it comes to privacy, there is no mandated lower bound for the physical distance between people, but it is deliberately maximized by each personal choice.
In particular, this can be seen at a row of payphones (Figure~\ref{fig:payphones}), where people tend to occupy an available payphone that is most distant from those already occupied.
A similar trend can be observed at a row of urinals in a male restroom, and it was even stated as one of the rules in the \emph{Male Restroom Etiquette} mockumentary by Phil Rice~\cite{Rice2006wiki}.

\begin{figure}[!ht]
\centering \includegraphics[width=0.75\textwidth]{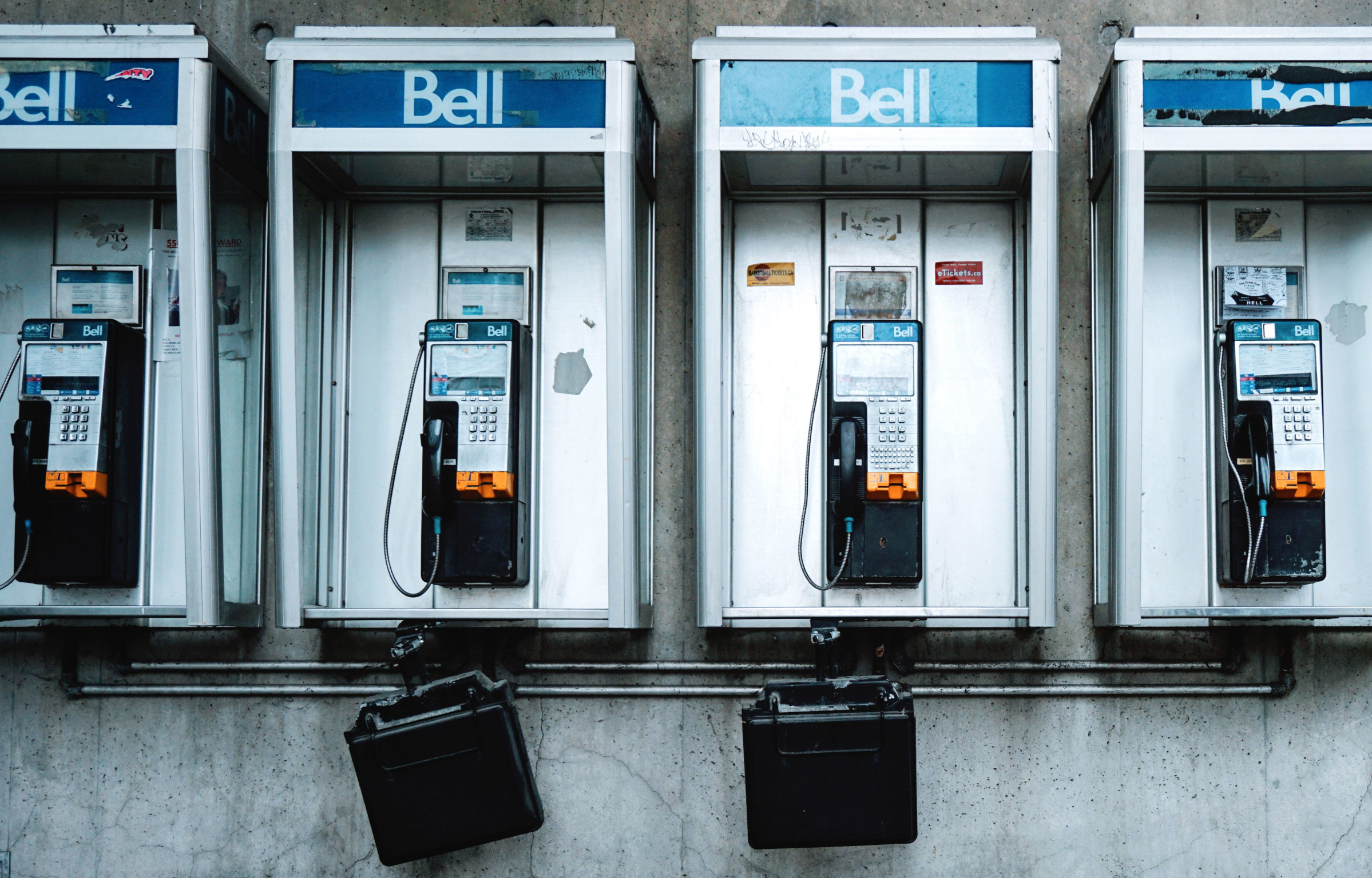}
\caption{A row of payphones. (pxfuel.com)}
\label{fig:payphones}
\end{figure}

Assuming that there are $n$ payphones in a row and that $n$ people occupy payphones one after another as privately as possible, the resulting assignment of people to payphones defines a permutation, 
which we will refer to as a \emph{payphone permutation}.

For example, when $n=6$ we have six people numbered from $1$ to $6$ and a row of six payphones:
\begin{center}
\begin{tabular}{lcc}
people & ~~~~~~ & payphones\\
\\
6, 5, 4, 3, 2, 1 & ~~~$\longrightarrow$~~~ & 
\begin{tabular}{|c|c|c|c|c|c|}
\hline
\phantom{3} & \phantom{5} & \phantom{1} & \phantom{4} & \phantom{6} & \phantom{2} \\
\hline
\end{tabular}\,.
\end{tabular}
\end{center}

\noindent
The first person does not have any preference and takes any available payphone, let it be the third one:
\begin{center}
\begin{tabular}{lcc}
6, 5, 4, 3, 2\phantom{, 1} & ~~~$\longrightarrow$~~~ & 
\begin{tabular}{|c|c|c|c|c|c|}
\hline
\phantom{3} & \phantom{5} & 1 & \phantom{4} & \phantom{6} & \phantom{2} \\
\hline
\end{tabular}\,.
\end{tabular}
\end{center}

\noindent
Then the second person takes the most distant (last) payphone from the occupied one:
\begin{center}
\begin{tabular}{lcc}
6, 5, 4, 3\phantom{, 2, 1} & ~~~$\longrightarrow$~~~ & 
\begin{tabular}{|c|c|c|c|c|c|}
\hline
\phantom{3} & \phantom{5} & 1 & \phantom{4} & \phantom{6} & 2 \\
\hline
\end{tabular}\,.
\end{tabular}
\end{center}

\noindent
Now, every available payphone except the first one is adjacent to an occupied one, and thus the third person takes the first payphone:
\begin{center}
\begin{tabular}{lcc}
6, 5, 4\phantom{, 3, 2, 1} & ~~~$\longrightarrow$~~~ & 
\begin{tabular}{|c|c|c|c|c|c|}
\hline
3 & \phantom{5} & 1 & \phantom{4} & \phantom{6} & 2 \\
\hline
\end{tabular}\,.
\end{tabular}
\end{center}

\noindent
The fourth person cannot avoid taking a payphone adjacent to an occupied one, but the second payphone is sandwiched between two occupied ones,
while the fourth and fifth payphones each have only one adjacent occupied payphone. Hence, they may be more preferable, and let's assume 
the fourth person takes the fourth payphone:
\begin{center}
\begin{tabular}{lcc}
6, 5\phantom{, 4, 3, 2, 1} & ~~~$\longrightarrow$~~~ & 
\begin{tabular}{|c|c|c|c|c|c|}
\hline
3 & \phantom{5} & 1 & 4 & \phantom{6} & 2 \\
\hline
\end{tabular}\,.
\end{tabular}
\end{center}

\noindent
The fifth person has two equally uncomfortable options, each sandwiched between two occupied payphones, and can take any one of them. Let it be the second payphone:
\begin{center}
\begin{tabular}{lcc}
6\phantom{, 5, 4, 3, 2, 1} & ~~~$\longrightarrow$~~~ & 
\begin{tabular}{|c|c|c|c|c|c|}
\hline
3 & 5 & 1 & 4 & \phantom{6} & 2 \\
\hline
\end{tabular}\,.
\end{tabular}
\end{center}

\noindent
Finally, the sixth person takes the last available (fifth) payphone:
\begin{center}
\begin{tabular}{lcc}
\phantom{6, 5, 4, 3, 2, 1} & ~~~\phantom{$\longrightarrow$}~~~ & 
\begin{tabular}{|c|c|c|c|c|c|}
\hline
3 & 5 & 1 & 4 & 6 & 2 \\
\hline
\end{tabular}\,.
\end{tabular}
\end{center}
The resulting assignment of people to payphones defines a payphone permutation $(3,5,1,4,6,2)$.
It can be seen that not all permutations can be obtained this way. 
In the present study, we consider different types of payphone permutations and enumerate them.

\section{Basic Types of Payphone Permutations}

Consider a row of payphones, some of which may be occupied.
A maximal sequence of available consecutive payphones is called an \emph{interval}. An interval is called \emph{semi-closed} or \emph{closed} if it is flanked by $1$ or $2$ occupied payphones, respectively.
The \emph{length} of an interval equals the number of payphones in it. 
We find it convenient to view a semi-closed interval $I$ of length $\ell$ as just a visible ``half'' of an artificial closed interval of length $2\ell-1$,
which we refer to as the \emph{full length} of $I$. For a closed interval, the full length is identical to the length. 
A closed interval of length $1$ is composed of a single payphone, which we will call \emph{sandwiched} between two occupied payphones.

We will distinguish the following types of payphone permutations, in which the first person can take any payphone but the follow-up people obey one of the rules:
\begin{description}
\item[(P1)] A person takes any payphone such that its closest occupied payphone is as distant as possible.
\item[(P2)] A person first selects a longest interval with respect to the full length and then takes a phone within it as in (P1). 
\end{description}

Rule (P1) concerns only the closest occupied payphone and requires it to be as far away as possible. 
This rule implies that the closed interval of length $2k+1$ (in which a new person would be at distance $k+1$ from each of the flanking occupied payphones) 
is always preferred over the closed interval of length $2k$ (in which a new person would be at distances $k$ and $k+1$ from the flanking occupied payphones).
However, this rule does not give any preference for the latter interval over the closed interval of length $2k-1$ (in which a new person would be at distance $k$ from each of the flanking occupied payphones).

Rule (P2) additionally enforces that length $2k$ is preferred over length $2k-1$, which can be interpreted as maximizing the \emph{average distance} to the flanking occupied payphones,
since the corresponding average distances for the two lengths are $k+\frac{1}{2}$ and $k$, respectively.

We note that the permutation $(3,5,1,4,6,2)$ discussed in the previous section is a payphone permutation under each of the rules (P1) and (P2).
However, a permutation $(3,4,1,5,6,2)$, where the person $4$ takes a sandwiched payphone in the presence of a closed interval of length $2$, is a payphone permutation only under rule (P1).

Rules (P1) and (P2) also have circular variants, referred to as (C1) and (C2), in which the payphones form a circle rather than a row.
In these cases, no semi-closed intervals can appear, making their analysis a bit easier.

In Section~\ref{sec:other}, we will consider additional types of payphone permutations that alter preferences for payphones adjacent to an occupied one in some popular ways.

\section{Evolution of Intervals}

Suppose that we have a closed interval of length $n$ and that $n$ people come and take payphones according to rule (P1).
Let $S(n)$ be the multiset of the lengths of the nonempty intervals that ever appear in this process.
Since the first person breaks the initial interval into two of lengths $\Floor{\frac{n-1}2}$ and $\Ceil{\frac{n-1}2}$, we get a recurrence formula:
$$
S(n) = \begin{cases}
\emptyset, & \text{if}\ n=0;\\
\{n\} \cup S(\Floor{\frac{n-1}2}) \cup S(\Ceil{\frac{n-1}2}), & \text{if}\ n\geq 1.
\end{cases}
$$
For example:
\[
\begin{split}
S(1) &= \{ 1 \},\\
S(2) &= \{ 1, 2 \},\\
S(3) &= \{ 1^2, 3\},\\
S(4) &= \{ 1^2, 2, 4\}.
\end{split}
\]

In the next section, we will see that multisets $S(n)$ play a crucial role in the enumeration of payphone permutations. However, the recurrence formula is not convenient for computation.
Therefore, we start by deriving an explicit formula for the multiplicity of $k$ in $S(n)$, denoted $\#_k S(n)$.

First, we notice that $\#_k S(n)$ equals the number of walks from $k$ to $n$ in the directed (multi)tree $T$, 
where nodes are labeled with the positive integers, node $2$ has an incoming edge from node $1$, and node $m\geq 3$ has two incoming edges (possibly parallel): 
from nodes $\Floor{\frac{m-1}2}$ and $\Ceil{\frac{m-1}2}$ (Figure~\ref{fig:treeT}).

\begin{figure}[!tb]
\centering \includegraphics[width=\textwidth]{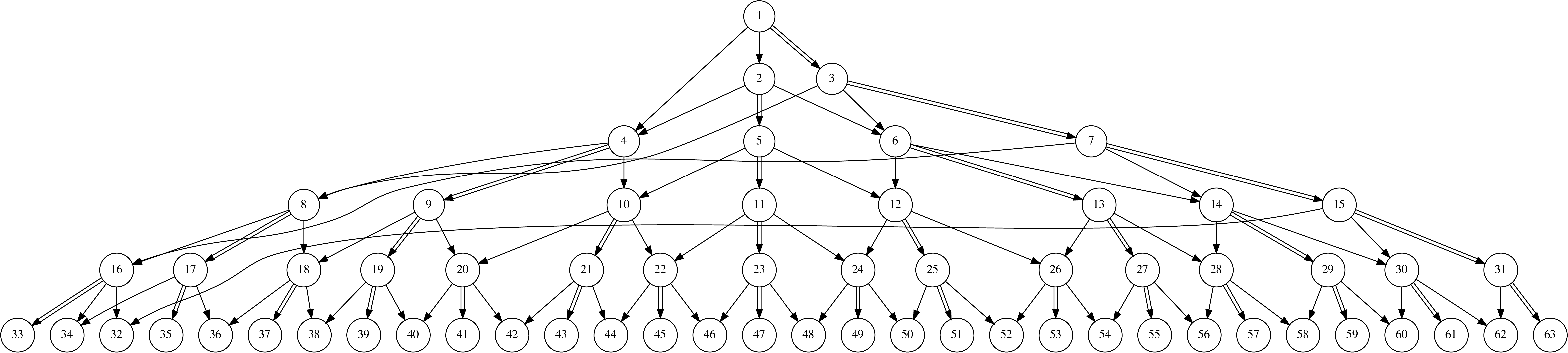}
\caption{The top part of the tree $T$ composed of nodes with labels from $1$ to $63$.}
\label{fig:treeT}
\end{figure}

\begin{lemma}\label{lem:Snk1} For any positive integers $n,k$,
$$\#_k S(n) = \sum_{\ell=\Ceil{\log_2\frac{n+2}{k+2}}}^{\Floor{\log_2\frac{n}{k}}} 2^\ell - |n+1 - (k+1)2^{\ell}|.$$
\end{lemma}

\begin{proof}
We find it convenient to work with a base-2 number system with digits from $Q:=\{0,1,2\}$, which we refer to as the \emph{quasi-binary number system}.
We will view elements of $Q^\ell$ as $\ell$-digit quasi-binary representations, and for $q=(q_1,\dots,q_\ell)\in Q^\ell$ define its numerical value
as $v(q):=\sum_{i=1}^\ell q_i 2^{\ell-i}$.

It is easy to see that every integer $k\geq 1$ in $T$ has four outgoing edges: one to $2k$, two to $2k+1$, and one to $2k+2$.
The numbers $2k$, $2k+1$, and $2k+2$ can be obtained from the binary representation of $k$ by appending a digit $0$, $1$, and $2$, respectively.
It follows that the number of walks from $k$ to $n$ of length $\ell$ can be calculated from the representations of $n$ as $n = k\cdot 2^{\ell} + v(q)$ with $q\in Q^\ell$.
Namely, each such representation corresponds to $2^{u(q)}$ walks, where $u(q)$ denotes the number of digits $1$ in $q$. That is,
$$\#_k S(n) = \sum_{\ell\geq 0}~~ \sum_{q\in Q^\ell:\ k\cdot 2^{\ell} + v(q) = n} 2^{u(q)}.$$

We further notice that for fixed $\ell$ and $m$, $\sum_{q\in Q^\ell:\ v(q)=m} 2^{u(q)}$ equals the number of pairs $(a,b)$ of $\ell$-bit nonnegative numbers such that $a+b=m$. 
Indeed, from each such pair, we can construct $q\in Q^\ell$ by performing carry-free bitwise addition of the binary representations of $a$ and $b$. Clearly, we have $a+b=v(q)=m$.
Vice versa, such pairs $(a,b)$ can be reconstructed from $q$ digit-wise as follows: if the $i$-th digit of $q$ is $0$, then the $i$-th bit in both $a$ and $b$ is $0$;
if the $i$-th digit of $q$ is $2$, then the $i$-th bit in both $a$ and $b$ is $1$; and if the $i$-th digit of $q$ is $1$, then the $i$-th bits in $a$ and $b$ are $0$ and $1$, or $1$ and $0$.
Since only in the last case we have two variations, the total number of such pairs reconstructed from $q$ equals $2^{u(q)}$, proving the formula.

Now let's take $m := n - k\cdot 2^{\ell}$.
From the constructed correspondence, it follows that $\sum_{q\in Q^\ell:\ v(q)=m} 2^{u(q)}$ equals the number of integers $a$ such that $0\leq a<2^{\ell}$ and $0\leq m-a<2^{\ell}$. 
Equivalently, we have $\max\{0,m+1-2^{\ell}\}\leq a < \min\{2^{\ell},m+1\}$ and thus
\[
\begin{split}
\sum_{q\in Q(\ell):\ v(q)=m} 2^{u(q)}
&= \min\{2^{\ell},m+1\} - \max\{0,m+1-2^{\ell}\} \\
&= 2^{\ell} - |m+1-2^{\ell}| \\
&= 2^{\ell} - |n+1-(k+1)2^{\ell}|.
\end{split}
\]
From the definition of $m$, it further follows that $0 \leq n - k\cdot 2^{\ell} \leq 2(2^\ell - 1)$, implying that $\ell$ satisfies the inequality $\frac{n+2}{k+2} \leq 2^\ell \leq \frac{n}{k}$.
\end{proof}

\begin{theorem}\label{th:kSn} 
For any positive integers $n\geq k\geq 2$, we have
$$
\#_k S(n) = 
\begin{cases}
1 + (n\bmod 2^\ell), & \text{if}\ k=\Floor{\frac{n}{2^\ell}}\ \text{for}\ \ell\in\{0,1,\dots,\Floor{\log_2 n}-1\}, \\
2^\ell - 1 - (n\bmod 2^\ell), & \text{if}\ k=\Floor{\frac{n}{2^\ell}}-1\ \text{for}\ \ell\in\{0, 1,\dots,\Floor{\log_2 \frac{n}3}\}, \\
0, & \text{otherwise},
\end{cases}
$$
and
$$
\#_1 S(n) = 
\begin{cases}
1 + (n\bmod 2^\ell), & \text{if}\ n\geq 3\cdot 2^{\ell-1}-1\ \text{for}\ \ell=\Floor{\log_2 n}, \\
2^{\ell-1}, & \text{if}\ n < 3\cdot 2^{\ell-1}-1\ \text{for}\ \ell=\Floor{\log_2 n}.
\end{cases}
$$
\end{theorem}

\begin{proof}
Notice that for $k\geq 2$, there exists at most one power of $2$ in the interval $\big[\frac{n+2}{k+2}, \frac{n}{k}\big]$ since
$$\frac{n}{k}\,\big/\,\frac{n+2}{k+2} < \frac{k+2}{k} \leq 2.$$ 
Suppose that $2^\ell\in\big[\frac{n+2}{k+2}, \frac{n}{k}\big]$ for some integer $\ell\geq 0$. 
Consider two cases depending on whether $2^\ell \geq \frac{n+1}{k+1}$.

If $\frac{n+1}{k+1}\leq 2^\ell \leq \frac{n}{k}$, then
$$\frac{n-(2^\ell-1)}{2^\ell} \leq k \leq \frac{n}{2^\ell}.$$
In this case, all different values of $k\geq 2$ are given by $k = \Floor{\frac{n}{2^\ell}} = \frac{n-(n\bmod 2^\ell)}{2^\ell}$ for $\ell\in\{0,1,\dots,\Floor{\log_2 n}-1\}$, for which we have $\#_k S(n) = 1 + (n\bmod 2^\ell)$ by Lemma~\ref{lem:Snk1}.

If $\frac{n+2}{k+2}\leq 2^\ell < \frac{n+1}{k+1}$, then
$$\frac{n+1-(2^\ell-1)}{2^\ell} \leq k+1 < \frac{n+1}{2^\ell}.$$
This inequality is soluble for $k$ if and only if $2^\ell\nmid n+1$, in which case $k = \Floor{\frac{n+1}{2^\ell}}-1 = \Floor{\frac{n}{2^\ell}}-1 = \frac{n-(n\bmod 2^\ell)}{2^\ell}-1$. 
We have with necessity $\ell>d:=\nu_2(n+1)$ (the $2$-adic valuation of $n+1$) and thus all different values of $k\geq 2$ are delivered by integers $\ell$ satisfying
$d<\ell\leq \Floor{\log_2 \frac{n}3}$, for which we have
$\#_k S(n) = 2^\ell - 1 - (n\bmod 2^\ell)$ by Lemma~\ref{lem:Snk1}.
However, since this formula gives $0$ whenever $2^\ell\mid n+1$, we can drop the requirement $\ell>d$.

When $k=1$, the interval $\big[\frac{n+1}{k+1}, \frac{n}{k}\big]$ contains $2^\ell$ with $\ell := \Floor{\log_2 n}$, 
and there is also the possibility that the interval $\big[\frac{n+2}{k+2}, \frac{n+1}{k+1}\big)$ contains $2^{\ell-1}$.
This happens when $n<3\cdot 2^{\ell-1}-1$, or equivalently $n\bmod 2^\ell < 2^{\ell-1}-1$, which by combining the formulas from the above two cases implies that
$$\#_1 S(n) = 1 + (n\bmod 2^{\ell}) + 2^{\ell-1} - 1 - (n\bmod 2^{\ell-1}) = 2^{\ell-1}.$$
Otherwise, when $n\geq 3\cdot 2^{\ell-1}-1$, the formula is the same as in the first case above, that is, $\#_1 S(n) = 1 + (n\bmod 2^\ell)$.
\end{proof}

\section{Enumeration of Payphone Permutations}

We will show that the enumeration of payphone permutations heavily relies on the multiplicities $\#_k S(n)$, which we can compute efficiently thanks to Theorem~\ref{th:kSn}.
We start with enumerating permutations of type (C2).

\begin{theorem}\label{th:C2}
Let $C_2(n)$ be the number of payphone permutations of type (C2) and size $n$.
Then for any $n\geq 1$,
$$C_2(n) = n\cdot F(S(n-1)),$$
where the function $F$ is defined on the finite multisets of positive integers as follows:
$$F(M) := \bigg(\prod_{k\geq 1} (\#_k M)!\bigg)\cdot 2^{\sum_{k\geq 1} \#_{2k} M}.$$
\end{theorem}

\begin{proof} 
Consider a circular booth of $n$ payphones.
The first person chooses any payphone,
which contributes the factor $n$ to the formula and results in a closed interval of length $n-1$.

According to rule (C2), people come into the longest intervals and break them into shorter ones.
It follows that when a person enters an interval of length $k$ for the first time, the number of such intervals is equal to $\#_k M$.
People have the freedom to choose the order of these intervals, contributing the factor $(\#_k M)!$ to the formula.
For each interval of even length, there is also the freedom to choose one of the two middle payphones, contributing a factor $2$.
The total number of such intervals is $\sum_{k\geq 1} \#_{2k} M$.
\end{proof}

\begin{theorem}\label{th:C1}
Let $C_1(n)$ be the number of payphone permutations of type (C1) and size $n$.
Then for any $n\geq 1$,
$$C_1(n) = n\cdot G(S(n-1)),$$
where the function $G$ is defined on the finite multisets of positive integers as follows:
$$G(M) := \bigg(\prod_{k\geq 1} (\#_{2k} M + \#_{2k-1} M)!\bigg) \cdot 2^{\sum_{k\geq 2} \#_{2k} M}.$$
\end{theorem}

\begin{proof}
Rule (C1) gives equal preference to intervals of lengths $2k$ and $2k-1$. For $k>1$, these intervals are broken into intervals shorter than $2k-1$.
Therefore, for $k>1$, the order in which people break these intervals contributes the factor $(\#_{2k} M + \#_{2k-1} M)!$ to the formula.

The case $k=1$ is trickier, since each interval of length $2$ is broken into an interval of length $1$. The latter interval can be selected only after the former one has been selected.
Hence, in the ordered sequence of intervals of lengths $1$ and $2$, each interval of length $2$ must come before the corresponding interval of length $1$, which
reduces the total number of orderings $(\#_{2} M + \#_{1} M)!$ by the factor $2^{\#_2 M}$.
Also, as before, each interval of even length here contributes a factor $2$, which totals the factor $2^{\sum_{k\geq 1} \#_{2k} M}$.
\end{proof}

\begin{theorem}\label{th:P2}
Let $P_2(n)$ be the number of payphone permutations of type (P2) and size $n$.
Then for any $n\geq 1$,
$$P_2(n) = \sum_{i=1}^n F(S'(i-1)\cup S'(n-i)),$$
where
$$
S'(\ell) :=
\begin{cases}
\emptyset, & \text{if}\ \ell=0; \\
S(\ell-1) \cup \{2\ell-1\}, & \text{if}\ \ell\geq 1.
\end{cases}
$$
\end{theorem}

\begin{proof}
Consider a row of $n$ payphones.
The first person can choose the $i$-th payphone for any $i\in\{1,2,\dots,n\}$, breaking the row into two semi-closed (possibly empty) intervals of length $i-1$ and $n-i$.

Notice that a first person coming to a semi-closed interval $I$ of length $\ell\geq1$ (for $\ell = i-1$ or $n-i$) always takes a payphone at its open end, which we view as
the middle of the corresponding artificial interval of length $2\ell-1$. This is the full length of $I$ according to which $I$ is compared to the other intervals under rule (P2).
Taking a payphone at the open end of $I$ transforms it into a closed interval of length $\ell-1$.
Therefore, the contribution of $I$ to the intervals evolution equals
$S'(\ell) = S(\ell-1) \cup \{2\ell-1\}$, where
the elements of $S(\ell-1)$ correspond to the intervals resulting from the evolution of the closed interval of length $\ell-1$.
Knowing the intervals evolution, we proceed similarly to the proof of Theorem~\ref{th:C2} to obtain the formula for $P_2(n)$.
\end{proof}

\begin{theorem}\label{th:P1}
Let $P_1(n)$ be the number of payphone permutations of type (P1) and size $n$.
Then for any $n\geq 1$,
$$P_1(n) = \sum_{i=1}^n G(S'(i-1)\cup S'(n-i)).$$
\end{theorem}

\begin{proof}
The proof combines the ideas of the proofs of Theorems~\ref{th:C1} and \ref{th:P2}.
\end{proof}

\section{Other Types of Payphone Permutations}\label{sec:other}

We will consider some additional types of payphone permutations:
\begin{description}
\item[(P3)] A person follows rule (P1), except that
among the available payphones adjacent to an occupied one the most preferred are those located at the open ends (when available).
\item[(P4)] A person follows rule (P2), except that a payphone adjacent to a single occupied payphone is preferred to a sandwiched payphone.
\item[(P5)] A person follows rule (P1), except that a payphone adjacent to a single occupied payphone is preferred to a sandwiched payphone.
\end{description}

Compared to rules (P1) and (P2), rules (P3)---(P5) alter the treatment of the following intervals:
(i) a closed interval of length $1$ (a sandwiched payphone); (ii) a closed interval of length $2$; and (iii) a semi-closed interval of length $1$.
Each of such intervals is composed of payphones adjacent to at least one occupied payphone, which therefore are indistinguishable from the perspective of the distance to their closest occupied payphones.
While for (P1) all three cases have the same preference, (P3) favors (iii) over equally preferred (i) or (ii); and (P5) equally prefers (ii) and (iii) over (i).
Similarly, while rule (P2) prefers (ii) over each of equally preferred (i) and (iii), rule (P4) additionally imposes the preference of (iii) over (i).
We remark that an interval of type (i) 
can be created only by the first person, since any following people coming to a semi-closed interval will turn it into a closed one.

We enumerate the payphone permutations of types (P3), (P4), and (P5) in Theorems~\ref{th:P3}, \ref{th:P4}, and \ref{th:P5} below.

\begin{theorem}\label{th:P3}
Let $P_3(n)$ be the number of payphone permutations of type (P3) and size $n$.
Then $P_3(1) = 1$, $P_3(2) = 2$, $P_3(3) = 4$, and for any $n\geq 4$,
$$P_3(n) = 2G(S'(n-1)) + 2G(S'(n-2)) + \sum_{i=3}^{n-2} G(S'(i-1)\cup S'(n-i)).$$
\end{theorem}

\begin{proof}
Let $n\geq 4$.
Rule (P3) is different from (P1) only if there exists a semi-closed interval $I$ of length $1$. 
This interval can be created only by the first person, which occurs for $i=2$ and $i=n-1$ in the formula for $P_1(n)$.
The terms corresponding to the other values of $i$ are inherited by the formula for $P_3(n)$ as is.

The terms corresponding to $i=2$ and $i=n-1$ in the formula for $P_1(n)$ both equal $G(S'(1)\cup S'(n-2))=G(\{1\}\cup S'(n-2))$,
which accounts for the contribution of $I$ as having the same preference as any other interval of length $1$.
However, under rule (P3), the interval $I$ must be selected first among the intervals of length $1$,
and thus the number of selection orders of such intervals does not depend on $I$. 
That is, the terms corresponding to $i=2$ and $i=n-1$ in the formula for $P_3(n)$ become simply $G(S'(n-2))$.
\end{proof}

\begin{theorem}\label{th:P4}
Let $P_4(n)$ be the number of payphone permutations of type (P4) and size $n$.
Then $P_4(1) = 1$, $P_4(2) = 2$, $P_4(3) = 4$, and for any $n\geq 4$,
$$P_4(n) = 2F(S'(n-1)) + 2F(S'(n-2)) + \sum_{i=3}^{n-2} F(S'(i-1)\cup S'(n-i)).$$
\end{theorem}

\begin{proof} 
The proof is similar to the proof of Theorem~\ref{th:P3} adjusting the formula for $P_1(n)$, but applies to the formula for $P_2(n)$.
\end{proof}

\begin{theorem}\label{th:P5}
Let $P_5(n)$ be the number of payphone permutations of type (P5) and size $n$.
Then $P_5(1) = 1$, $P_5(2) = 2$, $P_5(3) = 4$, and for any $n\geq 4$,
$$P_5(n) = 2H(S'(n-1)) + H(\{2\}\cup S'(n-2)) + \sum_{i=3}^{n-2} H(S'(i-1)\cup S'(n-i)),$$
where the function $H$ is defined on the finite multisets of positive integers as follows:
$$H(M) := \bigg(\prod_{k\geq 2} (\#_{2k} M + \#_{2k-1} M)!\bigg) \cdot (\#_2 M)!\cdot (\#_1 M)! \cdot 2^{\sum_{k\geq 1} \#_{2k} M}.$$
\end{theorem}

\begin{proof}
Let $n\geq 4$.
Rule (P5) is different from (P1) only when it comes to accounting for intervals of full length $1$ and $2$.
The formula for $P_5(n)$ is therefore obtained from the one for $P_1(n)$ by replacing the function $G$ with a new function $H$
having the same terms involving $\#_{2k} M$ and $\#_{2k-1} M$ for $k\geq 2$, but differing for $k=1$.
Namely, while the number of possible selection orders for the intervals of length $1$ and $2$ under rule (P1) equals $(\#_{2} M + \#_{1} M)! / 2^{\#_2 M}$ (as explained in the proof of Theorem~\ref{th:C1}),
under rule (P5) this number is simply $(\#_{2} M)!\cdot (\#_{1} M)!$ since the intervals of length $2$ must be selected before the intervals of length $1$.

We also need to take special care of the semi-closed interval of length $1$, which appears in the cases of $i=2$ and $i=n-2$. These cases together contribute $2G(S'(1)\cup S'(n-2)) = 2G(\{1\}\cup S'(n-2))$ to the formula for $P_1(n)$.
However, under rule (P5), such an interval has the same preference as any closed interval of length $2$, suggesting modification of the corresponding term to $2H(\{2\}\cup S'(n-2))$ in the formula for $P_5(n)$. It remains only to divide this term by $2$ since, unlike a closed interval of length $2$ providing two ways to break it, there is only one way for a semi-closed interval of length $1$.
\end{proof}

\section{Discussion}
The Online Encyclopedia of Integer Sequences~\cite{OEIS}
contains sequences formed by the numbers of payphone permutations of each type.
The following table provides the corresponding indices and the first ten terms of each sequence.

\begin{center}
\begin{tabular}{c|c|cccccccccc}
OEIS index & $n$  &  1 & 2 & 3 & 4 & 5 & 6 & 7 & 8 & 9 & 10 \\
\hline\hline
\texttt{A358056} & $P_1(n)$ & 1 & 2 & 4 & 8 & 20 & 48 & 216 & 576 & 1392 & 7200 \\
\hline
\texttt{A095236} & $P_2(n)$ & 1 & 2 & 4 & 8 & 16 & 36 & 136 & 216 & 672 & 2592 \\
\hline
\texttt{A361295} & $P_3(n)$ & 1 & 2 & 4 & 6 & 12 & 40 & 144 & 384 & 1008 & 6816 \\
\hline
\texttt{A095912} & $P_4(n)$ & 1 & 2 & 4 & 6 & 12 & 28 & 104 & 152 & 528 & 2208 \\
\hline
\texttt{A363785} & $P_5(n)$ & 1 & 2 & 4 & 6 & 16 & 28 & 120 & 264 & 576 & 2784 \\
\hline
\texttt{A361296} & $C_1(n)$ & 1 & 2 & 6 & 8 & 60 & 144 & 336 & 384 & 8640 & 57600\\
\hline
\texttt{A095239} & $C_2(n)$ & 1 & 2 & 6 & 8 & 40 & 96 & 168 & 384 & 1728 & 15360\\
\hline
\end{tabular}
\end{center}

The formulas established in Theorems~\ref{th:C2}--\ref{th:P5} can be used to compute hundreds of terms of these sequences.

Quite counterintuitively, the number of payphone permutations of type (C2) is not an increasing function:
we have $C_2(n) > C_2(n+1)$ for $n = 24, 32, 48, 56, 64, \dots$, and even $C_2(96) > C_2(97) > C_2(98)$.
Also, it is worth noticing that while $P_4(5) < P_5(5)$, we have $P_4(6)=P_5(6)=28$ and, moreover, they enumerate the same set of $28$ permutations.

We note that the function $f(n) := \#_1 S(n-1)$ satisfies the recurrence: $f(1)=0$, $f(2)=1$, and $f(n) = f(\Floor{\tfrac{n}2}) + f(\Ceil{\tfrac{n}2})$ for $n\geq 3$. It is present in the OEIS as the sequence {\tt A060973}, among many similar \emph{divide-and-conquer} recurrence sequences~\cite{Stephan2004}.
The explicit formula given in Theorem~\ref{th:kSn} makes $f(n)$ a nice test case for the asymptotic analysis developed in~\cite{Hwang2017}.

\section*{Acknowledgments}

The author thanks Eli Bagno and anonymous reviewers for providing valuable comments on an earlier version of this manuscript.

\bibliographystyle{plain}
\bibliography{payphone}

\end{document}